\documentclass[draft]{amsart}

\usepackage{amssymb}
\usepackage{url}
\usepackage{amscd}
\usepackage[all]{xy}

\def \N{{\mathbb N}}

\def \R{{\mathbb R}}

\def \1{{\mathbb 1}}

\theoremstyle{plain}
\newtheorem{theorem}{Theorem}
\newtheorem{proposition}{Proposition}
\newtheorem{definition}{Definition}
\newtheorem{lemma}{Lemma} 
\newtheorem{corollary}{Corollary}

\theoremstyle{remark}
\newtheorem{remark}{Remark}

\newtheorem{example}{Example}

\begin{document}


\title[Norm attaining  operators and variational principle]{Norm attaining operators and variational principle.}

\author{Mohammed Bachir}

\address{Laboratoire SAMM 4543, Universit\'e Paris 1 Panth\'eon-Sorbonne\\
Centre P.M.F. 90 rue Tolbiac\\
75634 Paris cedex 13\\
France}

\email{Mohammed.Bachir@univ-paris1.fr}



\begin{abstract} We establish a  linear variational principle extending the Deville-Godefroy-Zizler's one. We use this variational principle to  prove that if $X$ is a Banach space having property $(\alpha)$ of Schachermayer and $Y$ is any banach space, then the set of all norm strongly attaining linear operators from $X$ into $Y$ is a complement of a $\sigma$-porous set. Moreover, the results of the paper applies also to an abstract class of (linear and nonlinear) operator spaces.
\end{abstract}

\maketitle


\newcommand\sfrac[2]{{#1/#2}}

\newcommand\cont{\operatorname{cont}}
\newcommand\diff{\operatorname{diff}}



\section{\bf Introduction}
This paper is devoted to establish a new linear variational principle in the sprit of Stegall's one (see \cite{St} or \cite[Theorem 5.15]{Ph}),  which applies to a certain "small class" of subsets of Banach spaces. However, we do not need in our statment to assume that the Banach spaces have the Radon-Nikod\'ym property. The interest of this  result  is that, on the one hand, it extends the non-linear variational principle  of Deville-Godefroy-Zizler and Deville-revalski (see respectively \cite{DGZ} and \cite{DR}) and, on the other hand, it makes it possible to show that the set of norm attaining operators (under the hypothesis $(\alpha)$) is not only a dense subset of the space of all bounded linear operators but it is larger in the sense that is  a  complement of a $\sigma$-porous subset.  Moreover, "norm attaining operators" is extended to "strongly norm attaining operators".
\vskip5mm
Let $X$ and $Y$ be  real Banach spaces. The space $B(X,Y)$ (resp. the spaces $K(X,Y), F(X,Y)$) denotes the space of all bounded linear operators (resp. the spaces of compact operators, finite-rank operators). An operator $T \in B(X,Y)$  is said to be norm attaining  (resp. norm strongly attaining) if there is an $x_0 \in S_X$ (the sphere of $X$) such that $\|T\|=\|T(x_0)\|$ (resp. $\|T(x_n)\|\to\|T\|=\|T(x_0)\|$ implies that $\|x_n-x_0\|\to 0$). We write $NAB(X, Y)$ to denote the set of norm-attaining operators in $B(X,Y)$. The question whether $NAB(X, Y)$ is norm dense in $B(X,Y)$, starts in 1961 with the works of  Bishop and Phelps \cite{Bi-Ph1, Bi-Ph2}, where they proved that  if $Y$  is one-dimensional then $NAB(X, Y)$ is norm dense in $X^*= B(X, Y)$  for all spaces $X$.  In 1963,  Lindenstrauss \cite{Lin}, showed that the Bishop-Phelps theorem is not longer true for linear operators and gave some partial positive results. He introduced  property $(\beta)$ and proved that if $Y$ has the property $(\beta)$, then for every Banach space  $X$, $NAB(X, Y)$ is dense in $B(X,Y)$. Partington proved in \cite{Pa} that every Banach space $Y$ can be renormed to have the property $(\beta)$.  Schachermayer \cite{Sch} introduced property $(\alpha)$ as a sufficient condition on a Banach space $X$ such that $NAB(X, Y)$ is dense in $B(X,Y)$ for every $Y$  and he showed that  every weakly compactly generated Banach space can be renormed with property $(\alpha)$. Several authors have contributed in this domain, extending these results in different ways. There exists also a "quantitative version" of the Bishop-Phelps-Bollobás \cite{Bo} theorem given by Acosta, Aron, García and Maestre in \cite{AAGM}. Several authors have proven similar results, replacing $B(X,Y)$ by other type of operator spaces.   For a complete story of contributions in this domain, we refer to \cite{Ac} and the references therein.

The contribution on the subject of norm attaining operators in this paper, consist on replacing the density norm-attaining operators by the complement of $\sigma$-porosity and by giving an unified and abstract class of (linear and nonlinear)  operator spaces  satisfying the "norm attaining operators property" (see Theorem \ref{NA}). In particular, we obtain the following results: 

$(1)$ If $X$ has property $(\alpha)$ (see Example B in Section \ref{S2} for the definition), then for every Banach space $Y$ and every closed subspace $R(X,Y)$ of $B(X,Y)$ contaning $F(X,Y)$, we have that $NAR(X, Y)$ (the subset  of norm-attaining operators in $R(X,Y)$) is a complement of a $\sigma$-porous subset of $R(X,Y)$. In fact, we prove the result for  norm  strongly attaining operators.

$(2)$ The results of the paper applies also to nonlinear operator spaces as the space of all bounded  continuous (resp. uniformly continuous) functions from a complete metric space into a Banach space, extending some real-valued results of  Coban, Kenderov and Revalski in \cite{CKR} (see also \cite{DR}),  to the vector-valued framework.  Forr another direction of  Lipschitz norm  attaining  functions, we refer to \cite{KMS} and \cite{G}.

\vskip5mm
This paper is organized as follows. In Section \ref{S1}, we introduce a crucial property in our results that we called  {\it "uniform separation property"} (in short, $\mathcal{USP}$). We then give some examples of sets satisfying this property. In Section \ref{S2}, we prove our version of  linear variational principle  (Theorem ~\ref{princ})  and its localised version (Theorem ~\ref{corprinc}). We also gives an extension of  Deville-Godefroy-Zizler variational principle as immediat consequence. In Section \ref{S3}, we will apply this new variational principle to obtain, the $\sigma$-porosity of the set of norm nonattaining  operators in Theorem \ref{NA} and its corollaries.

\section{The uniform separation property.}\label{S1}
In this section, we introduce the notion of {\it uniform separation property} and gives some examples. The variational principle given in this paper, applies for general pseudometric spaces for generalised lower semicontinuous functions. We first recall the following definition.
\begin{definition} Let $C$ be a nonempty set and $\gamma : C\times C\to \R^+$. We say that $\gamma$ is a pseudometric if 

$(1)$ $\gamma(x,x)= 0$, for all $x \in C$.

$(2)$ $\gamma(x,y)=\gamma(y,x)$, for all $x \in C$.

$(3)$ $\gamma(x,y)\leq \gamma(x,z)+ \gamma(z,y)$, for all $x,y,z \in C$.
\end{definition}
Unlike a metric space,  one may have $\gamma ( x , y ) = 0$  for distinct values $ x \neq y $ . A pseudometric induces an equivalence relation, that converts the pseudometric space into a metric space. This is done by defining $x\sim y$ if  $\gamma ( x , y ) = 0$.
Let $\Gamma_\gamma:  C\to  C/\sim $ the canonical surjection mapping and let 
$$d_\gamma(\Gamma_\gamma(x),\Gamma_\gamma(y)):=\gamma(x,y).$$
Then, $(C/\sim,d_\gamma)$ is a well defined metric space. We say that $(C,\gamma)$ is a complete pseudometric space, if $(C/\sim,d_\gamma)$ is a complete metric space.

\begin{definition} \label{USP}
Let $X$ be a Banach space, $C$ be a subset of the dual $X^*$ and $(C,\gamma)$ be a pseudometric space.  We say that  $(C,\gamma)$ has the weak$^*$-uniform separation property (in short $w^*\mathcal{USP}$) in $X^*$ if there exists $a>0$ such that for every $\varepsilon\in ]0,a]$, there exists $\varpi_C(\varepsilon)>0$ such that for every $p\in C$, there exists $x_{p,\varepsilon}\in B_X$ (the closed unit ball of $X$) such that
\begin{eqnarray*}
\langle p, x_{p,\varepsilon} \rangle - \varpi_C(\varepsilon) \geq \langle q, x_{p,\varepsilon} \rangle, \textnormal{ for all } q\in C \textnormal{ such that } \gamma(q,p) \geq \varepsilon.
\end{eqnarray*}
 If $C$ is a subset of a  Banach space $X$, we say that $(C,\gamma)$ has the $\mathcal{USP}$ in $X$ if $(C,\gamma)$ has the $w^*\mathcal{USP}$ in $X^{**}$, when $C$ is considered as a subset of the bidual $X^{**}$. 
\end{definition}

The function $\varpi_C$ will be called, the {\it modulus of uniform separation} of $(C,\gamma)$. If $x\in X$, by  $\hat{x}$ we denote the evaluation  map at $x$ given by $\hat{x}: x^*\mapsto \langle x^*, x \rangle$, for all $x^*\in X^*$.  

\begin{remark} $1)$  If $A\subset C$ and $(C,\gamma)$ has the $w^*\mathcal{USP}$ (resp. the $\mathcal{USP}$), then clearly $(A, \gamma)$ also has the $w^*\mathcal{USP}$ (resp. the $\mathcal{USP}$). 

$2)$ Two interesting cases coresponds to framework where $\gamma$ is the norm of $X^*$ or the distance associated to the weak-star topoloy if the space $X$ is separable, but working with the general pseudometric has its applications as we will see in the context of norm attaining linear operators.

\end{remark}

The following proposition is easy to establish, his proof is left to the reader.
\begin{proposition} \label{CUSP} Let $X$ be a Banach space and $C$ be a subset of $X^*$ (resp. subset of $X$). Suppose that $(\overline{C},\gamma)$ is a pseudometric space (where $\overline{C}$ denotes the norm closure of $C$) and  the identity map $i: (C,\|\cdot\|)\to(C,\gamma)$ is continuous. Then, $(C,\gamma)$ has the $w^*\mathcal{USP}$ (resp. has the $\mathcal{USP}$) if and only if $(\overline{C},\gamma)$ has the $w^*\mathcal{USP}$ (resp. has the $\mathcal{USP}$).
\end{proposition}

\subsubsection{\bf Examples of subsets having the $\mathcal{USP}$.} \label{examples}  We give some examples of sets satisfying the $\mathcal{USP}$ or $w^*\mathcal{USP}$.
\paragraph{\bf A.  Uniform convex spaces.} Recall that a Banach space $(L,\|.\|)$ is uniformly convex if for each $\varepsilon\in ]0,2]$,
\begin{eqnarray*}
\delta(\varepsilon):=\inf\lbrace 1-\|\frac{x+y}{2}\|: x,y\in S_L; \|x-y\|\geq \varepsilon\rbrace >0.
\end{eqnarray*}

\begin{proposition} \label{prop1} Let $L$ (resp. $L^*$) be a uniformly convex Banach space. Then the sphere $(S_L,\|.\|)$ (resp. the sphere $(S_{L^*},\|\cdot\|)$) has the $\mathcal{USP}$ in $L$ (resp. the $w^*\mathcal{USP}$ in $L^*$).
\end{proposition}
\begin{proof} Let $\varepsilon\in ]0, 2]$. For each $x, y\in S_L$ such that $\|x-y\|\geq \varepsilon$ we have 
\begin{eqnarray*}
\|\frac{x+y}{2}\| \leq 1- \delta(\varepsilon).
\end{eqnarray*}
Thus, for all $p\in S_{L^*}$ we have 
\begin{eqnarray*}
\langle p, \frac{x+y}{2}\rangle\leq \|\frac{x+y}{2}\| \leq 1- \delta(\varepsilon).
\end{eqnarray*}
Now, let us fix an arbitrary $x\in S_L$ and choose $p_{x,\varepsilon} \in S_{L^*}$ such that $\langle p_{x,\varepsilon}, x\rangle>1-\frac{\delta(\varepsilon)}{2}$. Using the above inequality, we get that
$\langle p_{x,\varepsilon}, y \rangle\leq 2-2\delta(\varepsilon)-\langle p_{x,\varepsilon}, x\rangle\leq \langle p_{x,\varepsilon}, x\rangle-\delta(\varepsilon)$ for all $y\in S_L$ such that $\|x-y\|\geq \varepsilon$. 
Hence, $(S_L,\|.\|)$ has the $\mathcal{USP}$ with modulus of uniform separation $\varpi_{S_L}(\varepsilon)=\delta(\varepsilon)$ for all $\varepsilon\in ]0,2]$ (the same proof work for $(S_{L^*},\|.\|)$) . 
\end{proof}
\paragraph{\bf B. Property $(\alpha)$.}  Recall the property $(\alpha)$ introduced by Schachermayer  (see \cite{Sch}).  A Banach space $X$ has property $(\alpha)$ if there exist $\lbrace x_\lambda: \lambda \in \Lambda \rbrace$, $\lbrace x^*_\lambda: \lambda \in \Lambda \rbrace$, subsets of $X$ and $X^*$  respectively, such that 

$1)$ $\|x_\lambda\|=\|x^*_\lambda\|=\langle x^*_\lambda, x_\lambda \rangle =1$ for all $\lambda \in \Lambda$.

$2)$ There exists a constant $\rho$ with $0<\rho < 1$ such that, for $\lambda, \mu \in \Lambda$ with $\lambda \neq \mu$, we have 
that $|\langle x^*_\lambda, x_\mu \rangle|\leq \rho.$

$3)$ The absolute convex hull of the set $\lbrace x_\lambda: \lambda \in \Lambda \rbrace$ is dense in the unit ball of $X$.

\vskip5mm
Clearly, conditions $1)$ and $2)$ implies that $(\lbrace x_\lambda: \lambda \in \Lambda \rbrace, \|\cdot\|_X)$ has the $\mathcal{USP}$ in $X$ and also $(\overline{\lbrace x_\lambda: \lambda \in \Lambda \rbrace}, \|\cdot\|_X)$ has the $\mathcal{USP}$ in $X$ by Proposition \ref{CUSP}. 
\vskip5mm

\paragraph{\bf C. The Dirac measures.} Let $(L,d)$ be a metric space and $(X,\|.\|_X)$ be a Banach space included in $C_b(L)$ (the space of all real-valued bounded continuous functions equipped with the sup-norm). Suppose that $X$ separates the points of $L$ and satisfies $\|.\|_X\geq \alpha\|.\|_{\infty}$ on $X$, for some $\alpha>0$. Recall that the Dirac measure associated to the point $x\in L$ is the evaluation  linear continuous functional $\delta_x: h\mapsto h(x)$, $h\in X$. Since $\|.\|_X\geq \alpha\|.\|_{\infty}$, it follows that $\|\delta_x\|\leq \frac{1}{\alpha}$ for all $x\in L$. Thus, the subset $\delta(L):=\lbrace \delta_x : x\in L\rbrace$ is norm bounded in $X^*$. We equipp the set $\delta(L)$ with the following complete metric :
$$\tilde{d}(\delta_x, \delta_y):=d(x,y).$$  
Notice that the map $\tilde{d}$ is well defined since $X$ separates the points of $L$. Let $h$ be a real-valued function on $L$ and $A$ be a subset of $L$. By $\textnormal{supp}(h):=\overline{\lbrace x\in L: h(x)\neq 0\rbrace}$, we denote the support of $h$ and by $\textnormal{diam}(A)$, we denote the diameter of $A$. We consider the following hypothesis: 

\noindent $ ({\bf H})$ : for every $\varepsilon>0$ there exists $\varpi_X(\varepsilon)>0$ such that, for every $x\in L$, there exists a function $b_{x,\varepsilon}\in B_X$ such that,
$$  b_{x,\varepsilon}(x)-\varpi_X(\varepsilon)\geq \sup_{y\in L:d(y,x)\geq \varepsilon} b_{x,\varepsilon}(y).$$

The following hypothesis is used by Deville-Revalski in \cite{DR}: 

\noindent ${\bf (DR)}$ for every natural number $n$, there exists a positive constant $M_n$ such that for any point $x \in L$ there exists a function $h_{x,n} : L \longrightarrow [0; 1]$, such that $h_{x,n} \in X$ ,
$\|h_{x,n}\|\leq M_n$, $h_{x,n}(x) = 1$ and $diam(supp (h_{x,n})) < \frac 1 n.$ 
\vskip5mm
Then, we have that: ${\bf (DR)} \Longrightarrow ({\bf H}) \Longleftrightarrow (\delta(L), \tilde{d}) \textnormal{ has the } w^*\mathcal{USP} \textnormal{ in } X^*.$
\vskip5mm
The fact that ${\bf (DR)} \Longrightarrow ({\bf H})$ is given by taking  $b_{x,\varepsilon}:=\frac{h_{x,[\frac 1 \varepsilon]+1}}{M_{[\frac 1 \varepsilon]+1}}\in B_{X}$ and $\varpi_X(\varepsilon)=\frac{1}{M_{[\frac 1 \varepsilon]+1}}$, for all $\varepsilon>0$, where $[\frac 1 \varepsilon]$ denotes the integer part of $\frac 1 \varepsilon$. The part $({\bf H}) \Longleftrightarrow (\delta(L), \tilde{d})$  has the  $w^*\mathcal{USP}$  in $X^*$, follows from the definitions. However, $({\bf H}) \not\Longrightarrow {\bf (DR)}$ in general. Indeed,  for a bounded complete metric space $(L,d)$, consider $(X\|.\|_X)=(\textnormal{Lip}_0(L),\|.\|)$, the space of all Lipschitz continuous functions that vanish at some point $x_0\in L$ equipped with its natural norm
$$\|g\|:=\sup_{x, y \in L:x\neq y} \frac{|g(x)-g(y)|}{d(x,y)}; \hspace{2mm} \forall g\in X.$$
Then, hypothesis $({\bf H})$  is  trivially satisfied with $\varpi_L(\varepsilon)=\varepsilon$ for all $\varepsilon>0$ and $b_{x,\varepsilon}(y):=d(x,x_0)-d(x,y)$ for all $x,y\in L$. However, hypothesis ${\bf (DR)}$  is never satified for $X=\textnormal{Lip}_0(L)$ since $f(x_0)=0$ for all $f\in \textnormal{Lip}_0(L)$. Thus, the condition that $(\delta(L), \tilde{d})$  has the  $w^*\mathcal{USP}$  in $X^*$ ($\Longleftrightarrow ({\bf H}) $) , is more general than the hypothesis ${\bf (DR)}$ used by Deville-Revalski in \cite{DR}.
\vskip5mm
The extension of the Deville-Revalski result in \cite{DR}, will be given by applying our main result (Theorem ~\ref{princ}) to the metric space $(\delta(L), \tilde{d})$ who has  the  $w^*\mathcal{USP}$.

\section{Linear variational principle.}\label{S2}

This section is devoted to establish a linear variational principle for $w^*\mathcal{USP}$ subsets of Banach spaces.  We recall that a function $f$ has a strong minimum on a metric space $(C,d)$ at some point $p\in C$, if $f$ attains its minimum at $p$ and for any sequence $(p_n)\subset C$ such that $f(p_n)\to f(p)=\inf_C f$, we have that $d(p_n,p)\to 0$. A function $f$ has a strong maximum if $-f$ has a strong minimum. To obtain our result in the more general case of pseudometric spaces, we need to introduce the following definition.
\begin{definition} \label{Gdir} Let $(C,\gamma)$  be a pseudometric space. Let $f: C \to \R \cup \lbrace +\infty \rbrace$ be a proper bounded from below function. We say that $f$ attains $\gamma$-strongly-directionally its infinimum over $C$ at a direction $u\in C$  if and only if for every sequence $(q_n)\subset C$ we have 
$$\lim_{n\rightarrow +\infty} f(q_n)=\inf_C f \Longrightarrow \lim_{n\rightarrow +\infty} \gamma(q_n,u)=0.$$
A function $g$ attains $\gamma$-strongly-directionally its supremum over $C$ iff $-g$ attains $\gamma$-strongly-directionally its infinimum over $C$.
\end{definition}
In the general case, it may be that in the previous definition we have that $\inf_C f \neq f(u)$. However, the direction $u$ is necessarilly unique up to the relation $\sim$, that is, every other direction $v\in C$ satisfying the above property is such that $\gamma(v,u)=0$ and the converse is also true. Note that if, moreover,  we assume that$f$ is lower semicontinuous with respect to the pseudometric $\gamma$ (that is, for every sequence $(q_n)\subset C$, $\liminf_{n\rightarrow +\infty} f(q_n)\geq f(u)$, whenever $\lim_{n\rightarrow +\infty} \gamma(q_n,u)=0$), then the infimum of $f$ is atained at $u$. In the particular case where $\gamma$  is a metric and $f$ is lower semicontinuous for $\gamma$, the $\gamma$-strongly-directionally infinimum coincides with the classical notion of {\it strong minimum} mentioned above.
\vskip5mm
We recall the notion of $\sigma$-porosity. In the following definition,  $\mathring{B}_X(x ; r)$ stands for the open ball in $X$ centered at $x$ and with radius $r > 0$.
\begin{definition}\label{prous}  Let $(X ; d)$ be a metric space and $A$ be a subset of $X$. The set $A$ is
said to be porous in $X$ if there exist $\lambda_0 \in (0; 1]$ and $r_0 > 0$ such that for any $x \in X$
and $r \in (0; r_0]$ there exists $y \in X$ such that $\mathring{B}_X(y; \lambda_0r) \subset \mathring{B}_X(x; r) \cap (X \setminus A)$. The set
$A$ is called $\sigma$-porous in $X$ if it can be represented as a countable union of porous sets in $X$.
\end{definition}
 Every $\sigma$-porous set is of first Baire category. Moreover, in $\R^n$, every $\sigma$-porous set is of Lebesque measure zero. However,  there does exist a non-$\sigma$-porous subset of $\R^n$ which is of the first category and of Lebesgue measure zero. For more informations about $\sigma$-porosity, we refer to \cite{Za}. 
\vskip5mm

We give now, the main results of this section. We will see in Corollary \ref{DRB}  (see below),  how to recover and extend easily the Deville-Godefroy-Zizler and Deville-Revalski variational principles, from the following theorem (a vector-valued variational principle of type  Deville-Godefroy-Zizler is also given in Theorem \ref{NA}). Note that changing the "infinimum" by "supremum" and $f$ by $-f$, we obtain the "supremum version" of the following theorem which will be used in the context of norm attaining operators.

\begin{theorem} \label{princ} Let $X$ be a Banach space and $C$ be a norm bounded subset of the dual $X^*$. Suppose that $(C,\gamma)$ is a complete pseudometric space having the $w^*\mathcal{USP}$ in $X^*$. Let $f: C \to \R\cup \lbrace +\infty \rbrace$ be any proper bounded from below function. Then, there exists a $\sigma$-porous subset $F$ of $X$ such that for every $x\in X\setminus F$, $f+\hat {x}$ attains  $\gamma$-strongly-directionally its infinimum over $C$ at some direction $u\in C$.
\end{theorem}
\begin{proof} 
For each $n\in \N^*$, let
\begin{eqnarray*}
O_n=\lbrace x\in X/ \exists p_n\in C: (f+\hat{x})(p_n)< \inf \lbrace(f+\hat{x})(p): p\in C ; \gamma(p,p_n)\geq \frac{1}{n}  \rbrace\rbrace
\end{eqnarray*}
Let us prove that $O_n$ is the complement of porous set in $X$. We prove that for each $n\in \N^*$, the requirements of Definition \ref{prous} is satified with an arbitrary $r_n>0$ and 
\begin{eqnarray}\label{lambda}
\lambda_n =\min (\frac 1 4, \frac{1}{8D}\varpi_C(\frac 1 n))),
\end{eqnarray}
where $D:=\sup_{p\in C}\|p\|$ and $\varpi_C(\cdot)$ is the modulus of $w^*\mathcal{USP}$ of $C$. Indeed, let $y\in X$ and $0<\varepsilon <r_n$, we want to find $y_n\in X$ such that 
$$\mathring{B}_X(y+y_n, \lambda_n\varepsilon)\subset \mathring{B}_X(y,\varepsilon)\cap O_n.$$ 
Let $p_n \in C$ such that 
\begin{eqnarray} \label{eqq1}
 (f+\hat{y})(p_n)\leq  \inf_C (f+\hat{y}) +\lambda_n\varepsilon D.
\end{eqnarray}
Since $(C,\gamma)$ has the $w^*\mathcal{USP}$ in $X^*$, there exists $x_n\in B_X$ such that
\begin{eqnarray*} \label{eqR}
\langle p_n, x_n \rangle - \varpi_C(\frac 1 n) \geq \sup_{p\in C: \gamma(p, p_n)\geq \frac 1 n} \langle p, x_n \rangle.
\end{eqnarray*}
Equivalently, multiplying by $\frac{-\varepsilon}{2} $, we have
\begin{eqnarray} \label{eqb}
 \langle p_n,\frac{-\varepsilon}{2} x_n\rangle \leq \inf_{ p\in C  ; \gamma(p,p_n)\geq \frac{1}{n}} \langle p,\frac{-\varepsilon}{2} x_n\rangle -\frac{\varepsilon}{2}\varpi_C(\frac 1 n)\\\nonumber
\end{eqnarray}
Let us set $y_n=\frac{-\varepsilon}{2} x_n$. We prove that $\mathring{B}_X(y+y_n, \lambda_n\varepsilon)\subset \mathring{B}_X(y,\varepsilon)\cap O_n.$
Indeed, the fact that $\mathring{B}_X(y+y_n, \lambda_n\varepsilon)\subset \mathring{B}_X(y,\varepsilon)$ is clear since $\|y_n\|\leq \frac{\varepsilon}{2}$ and $\lambda_n\leq \frac 1 4$. Let us prove that $\mathring{B}_X(y+y_n, \lambda_n\varepsilon)\subset O_n$. Let $z\in X$ such that $\|z\|< \lambda_n \varepsilon$. From (\ref{eqb}) and the definition of $\lambda_n$, we get that
\begin{eqnarray*} 
\langle p_n, z+ y_n\rangle &=& \langle p_n, \frac{-\varepsilon}{2} x_n\rangle + \langle p_n, z\rangle\nonumber\\
&<& \inf_{ p\in C ; \gamma(p,p_n)\geq \frac{1}{n}} \langle p, y_n\rangle - \frac{\varepsilon}{2}\varpi_C(\frac 1 n)+ \lambda_n \varepsilon D\nonumber\\
&<& \inf_{ p\in C ; \gamma(p,p_n)\geq \frac{1}{n}} \langle p, y_n\rangle - \frac{\varepsilon}{2}\varpi_C(\frac 1 n)+\frac{\varepsilon}{4}\varpi_C(\frac 1 n)\nonumber \\
&=& \inf_{ p\in C ; \gamma(p,p_n)\geq \frac{1}{n}} \langle p, y_n\rangle - \frac{\varepsilon}{4}\varpi_C(\frac 1 n) \nonumber\\
&\leq& \inf_{ p\in C ; \gamma(p,p_n)\geq \frac{1}{n}} \langle p, y_n\rangle -2\lambda_n \varepsilon D\nonumber\\
&\leq& \inf_{ p\in C ; \gamma(p,p_n)\geq \frac{1}{n}} \langle p, z+ y_n\rangle -\lambda_n \varepsilon D\nonumber\\
\end{eqnarray*}  
Thus, we have that
\begin{eqnarray} \label{eq22} 
\langle p_n, z+ y_n\rangle &<& \inf_{ p\in C ; \gamma(p,p_n)\geq \frac{1}{n}} \langle p, z+ y_n\rangle -\lambda_n \varepsilon D
\end{eqnarray} 
 Using (\ref{eqq1}) and (\ref{eq22}), we get
\begin{eqnarray*} \label{eq3} 
(f+\hat{y}+\hat{y}_n +\hat{z})(p_n)&=& (f+\hat{y})(p_n) +\langle p_n, z+ y_n\rangle\nonumber\\
                             &\leq& \inf_C (f+\hat{y}) +\lambda_n \varepsilon D + \langle p_n, z+ y_n\rangle\nonumber\\
                             &<& \inf_C (f+\hat{y}) +\inf_{ p\in C ; \gamma(p,p_n)\geq \frac{1}{n}} \langle p, y_n+z\rangle \nonumber\\
                             &\leq& \inf_{ p\in C ; \gamma(p,p_n)\geq \frac{1}{n}}(f+\hat{y})(p) + \inf_{ p\in C ; \gamma(p,p_n)\geq \frac{1}{n}} \langle p, y_n +z\rangle\nonumber\\
                             &\leq& \inf_{ p\in C ; \gamma(p,p_n)\geq \frac{1}{n}}(f+\hat{y}+\hat{y}_n + \hat{z})(p)
\end{eqnarray*}
This shows that $y+y_n+z\in O_n$ for all $\|z\|< \lambda_n \varepsilon$. Hence, $\mathring{B}_X(y+y_n, \lambda_n\varepsilon)\subset O_n$. Finally, we proved that $\mathring{B}_X(y+y_n, \lambda_n\varepsilon)\subset \mathring{B}_X(y, \varepsilon)\cap O_n$. Hence,  $O_n$ is the complement of porous set in $X$. Consequently, $\cap_{n\in \N} O_n$ is the complement of a $\sigma$-porous set in $X$. 

To concludes the proof, we need to show that for every $x\in \cap_{n\in \N} O_n$ (the $\sigma$-porous set is $F=X\setminus \cap_{n\in \N} O_n$), $f+\hat {x}$ attains  $\gamma$-strongly-directionally its infinimum over $C$ at some direction $u\in C$. Indeed, let $x\in \cap_{n\in \N} O_n$, then for each $n\geq 1$, there exists $p_n \in C$ such that
\begin{eqnarray*}
(f+\hat{x})(p_n)&<& \inf_{ q\in C ; \gamma(q,p_n)\geq \frac{1}{n}}(f+\hat{x})(q).
\end{eqnarray*}
First, we show that the sequence $(p_n)$ is Cauchy sequence in $(C,\gamma)$ for the pseudometric $\gamma$. Indeed, we have that for each $k>n$, $\gamma(p_k, p_n)< \frac{1}{n}$ (otherwise, by the definition of $p_n$, we have $(f+\hat{x})(p_n)< (f+\hat{x})(p_k)$ and since $\gamma(p_k, p_n) \geq \frac{1}{n} > \frac{1}{k}$, by the definition of $p_k$ we have $(f+\hat{x})(p_k)< (f+\hat{x})(p_n)$ which is a contradiction). Thus, $(p_n)$ is a Cauchy sequence in the complete pseudometric  space  $(C,\gamma)$ converging to some $u\in C$ ($u$ is unique up to the relation  $\sim$). Now, we prove that $f+\hat {x}$ attains  $\gamma$-strongly-directionally its infinimum over $C$ at the direction $u\in C$. Indeed,  let $(q_k)\subset C$ be any sequence such that $(f+\hat{x})(q_k)$ converges to $\inf_C(f+\hat{x})$. Suppose by contradiction that $(q_k)$ does not converges to $u$ for the pseudometric $\gamma$. Extracting if necessary a subsequence, we can assume that there exists $\varepsilon>0$ such that for all $k\in \N^*$, $\gamma(q_k, u)\geq\varepsilon$. Thus, there exists an integer $m$ such that $\gamma(q_k, p_m) \geq \frac 1 m$ for all $k\in \N^*$. It follows that, 
\begin{eqnarray*}
\inf_C(f+\hat{x})&\leq& (f+\hat{x})(p_m)\\
                          &<& \inf_{ q\in C ; \gamma(q, p_m) \geq \frac{1}{m}}(f+\hat{x})(q)\\
                          &\leq& (f+\hat{x})(q_k),
\end{eqnarray*}
for all $k\in \N^*$, which contradict the fact that $(f+\hat{x})(q_k)$ converges to $\inf_C (f+\hat{x})$. This ends the proof.
\end{proof} 

\vskip5mm
Now, we investigate the case where the pseudometric $\gamma$ is a metric. Typically, in the following corollary, the metric $d$ can be the norm of the dual space $X^*$ or a distance compatible with the weak-star topology if $C$ is weak-star metrizable subset of $X^*$. 
\begin{corollary} \label{CThm0} Let $X$ be a Banach space and $C$ be a norm bounded subset of $X^*$. Suppose that $(C,d)$ is a complete metric space such that the identity map $I_C: (C,d)\rightarrow (C, \textnormal{weak}^*)$ is continuous. Suppose that $(C,d)$ has the $w^*\mathcal{USP}$ in $X^*$. Let $f: (C,d) \to \R\cup \lbrace +\infty \rbrace$ be a proper bounded from below and lower semi-continuous. Then, there exists a $\sigma$-porous subset $F$ of $X$ such that for every $x\in X\setminus F$, $f+\hat{x}$ has a strong minimum on $(C,d)$.
\end{corollary}
\begin{proof} Since $I_C$ is $d$-to-weak-star continuous, then for every $x\in X$, $\hat{x}: x^*\mapsto \langle x^*, x \rangle$ is continuous on $C$ for the metric $d$. Thus, $f+\hat{x}$ is lower semicontinuous on $(C,d)$ and so we can apply Theorem \ref{princ} with the complete metric space $(C,\gamma)=(C,d)$, observing in this case that $\gamma$-strongly-directionnaly infinimum attaining for $f+\hat{x}$, consides with the notion of strong minimum for the distance $d$.
\end{proof}
As immediat application, we obtain the following extension of Deville-Revalski theorem in \cite{DR}.  Recall from Example C in Section \ref{S1}, that $ {\bf (DR)} \Longrightarrow ({\bf H})$ but $({\bf H}) \not\Longrightarrow {\bf (DR)}$ in general.

\begin{corollary} \label{DRB} Let $(L,d)$ be a complete metric space and $(X,\|.\|_X)$ be a Banach space included in $C_b(L)$ such that 

$(a)$ $\|.\|_X\geq \alpha\|.\|_{\infty}$ on $X$, for some $\alpha> 0$.

$(b)$ $X$ satisfies the hypothesis $({\bf H})$.

Let $f: L \to \R\cup \lbrace +\infty \rbrace$ be a proper bounded from below lower semi-continuous function. Then, there exists a $\sigma$-porous subset $F$ of $X$ such that for every $h\in X\setminus F$, $f+h$ has a strong minimum on $L$.
\end{corollary}
\begin{proof} We set $C:=\delta(L):=\lbrace \delta_x : x\in L\rbrace\subset X^*$. The hypothesis $({\bf H})$ is equivalent to the fact that $(C,\tilde{d})$ has the $w^*\mathcal{USP}$ in $X^*$, where $\tilde{d}(\delta_x, \delta_y):=d(x,y)$ is a complete metric space. On the other hand, it is trivial that the identity map $I_C: (C,\tilde{d})\to (C,\textnormal{weak}^*)$ is continuous (by the continuity of the elements of $X$ on $(L,d)$). We apply Corollary \ref{CThm0} to $(C,\tilde{d})$ and the proper bounded from below and lower semi-continuous  function $\tilde{f} : (C,\tilde{d})\to \R\cup \lbrace +\infty \rbrace$ defined by $\tilde{f}(\delta_x):=f(x)$ for all $x\in L$ (note that $\tilde{f}$ is well defined since $X$ separates the points of $L$, which is a consequence of hypothesis $({\bf H})$).
\end{proof}
\begin{remark} Note that in the dual space $(C_b(L))^*$, the set $(\delta(L),w^*)$ is completely metrizable by the metric $\tilde{d}$, but in general $(\overline{\delta(L)}^{w^*},w^*)$, which coincide (up to homeomorphism) with the Stone-Čech compactification $\beta L$ of $L$,  is not metrizable (if $(L,d)$ is not compact). 
\end{remark}
Now, we give in the following theorem, a localisation to Theorem \ref{princ}.
\begin{theorem} \label{corprinc}  Let $X$ be a Banach space and $C$ be a norm bounded subset of the dual $X^*$. Suppose that $(C,\gamma)$ is a complete pseudometric space having the $w^*\mathcal{USP}$ in $X^*$. Let $f: C \to \R\cup \lbrace +\infty \rbrace$ be any proper bounded from below function. Then, there exists $a>0$ such that for every $\varepsilon \in ]0,a]$ and every $p^*\in C$ such that $f(p^*)< \inf_C f + \varepsilon \varpi_C(\varepsilon)$ (where , $\varpi_C(\varepsilon)$ denotes the modulus of the $w^*\mathcal{USP}$ of $C$, in Definition \ref{USP}), there exists $x\in X$ and  $u \in C$ such that

$(i)$ $\gamma(p^*, u)\leq \varepsilon$,

$(ii)$ $\|x\| < 2 \varepsilon,$

$(iii)$ $f+\hat {x}$ attains  $\gamma$-strongly-directionally its infinimum over $C$ at the direction $u$. 
\end{theorem}
\begin{proof}
From the definition of the $w^*\mathcal{USP}$ (see Definition \ref{prous}), there exists $a>0$ such that for every $\varepsilon>0$, there exists $x_\varepsilon\in B_X$ such that 
$$(\bullet) \hspace{3mm} \langle p^*, x_\varepsilon\rangle -\varpi_C(\varepsilon) \geq \sup \lbrace \langle q, x_\varepsilon\rangle:  q\in C; \gamma(q, p^*) \geq \varepsilon\rbrace.$$
For evey $\theta>0$, let us set 
\begin{eqnarray*}
\lambda_{\varepsilon, \theta} &:=& (1+\theta)\frac{(f(p^*)-\inf_C f +\theta \varpi_C(\varepsilon))}{\varpi_C(\varepsilon)}.
\end{eqnarray*}
Then, clearly we have 
\begin{eqnarray}\label{lambda}
0<\lambda_{\varepsilon, \theta} &<& \frac{(1+\theta)(\varepsilon\varpi_C(\varepsilon) + \theta \varpi_C(\varepsilon)))}{\varpi_C(\varepsilon)}\\
&=& (1+\theta)(\varepsilon+ \theta )\nonumber
\end{eqnarray}
Now, we apply Theorem \ref{princ} to the function $h= f-\lambda_{\varepsilon, \theta}\hat{x}_\varepsilon$. Thus, there exists $y\in X$ and an element $u\in C$ such that $\|y\|<\frac{\theta\varpi_C(\varepsilon)}{2D}$ (where, $D:=\sup_{q\in C}\|q\|$) and $f-\lambda_{\varepsilon, \theta}\hat{x}_\varepsilon+\hat {y}$ attains $\gamma$-strongly-directionally its infinimum on $C$ at the direction $u$. Let us choose a sequence $(p_n)\subset C$ such that $$\lim_{n\rightarrow +\infty} (f-\lambda_{\varepsilon, \theta}\hat{x}_\varepsilon +\hat {y})(p_n)=\inf_C (f-\lambda_{\varepsilon, \theta}\hat{x}_\varepsilon+\hat {y}).$$ 

Then we have that,

$$(\bullet\bullet) \hspace{3mm}  \lim_{n\rightarrow+\infty} \gamma(p_n,u)=0. $$

On the other hand, we have that
\begin{eqnarray*}
\inf_C f +\liminf_{n\rightarrow+\infty} (-\lambda_{\varepsilon, \theta}\hat{x}_\varepsilon +\hat {y})(p_n) &\leq & \lim_{n\rightarrow +\infty} (f-\lambda_{\varepsilon, \theta}\hat{x}_\varepsilon +\hat {y})(p_n)\\
                                                &=& \inf_C (f-\lambda_{\varepsilon, \theta}\hat{x}_\varepsilon+\hat {y})\\
                                                &\leq& (f-\lambda_{\varepsilon, \theta}\hat{x}_\varepsilon +\hat {y})(p^*)
\end{eqnarray*}
Using the above inequality and the fact that $\|y\|<\frac{\theta \varpi_C(\varepsilon)}{2D}$ (where, $D:=\sup_{q\in C}\|q\|$), we get 
\begin{eqnarray*}
 \liminf_{n\rightarrow+\infty} -\lambda_{\varepsilon, \theta}\hat{x}_\varepsilon(p_n) &\leq & f(p^*)- \inf_C f -\lambda_{\varepsilon, \theta}\hat{x}_\varepsilon(p^*) +\theta\varpi_C(\varepsilon)
\end{eqnarray*}
Equivalently, 
\begin{eqnarray*}
 \liminf_{n\rightarrow+\infty}  \langle p^*- p_n, x_\varepsilon \rangle &\leq & \frac{f(p^*)- \inf_C f +\theta\varpi_C(\varepsilon)}{\lambda_{\varepsilon, \theta}}\\
 &=& \frac{1}{1+\theta}\varpi_C(\varepsilon).
\end{eqnarray*}
\noindent {\bf Claim.} We have that $\gamma(p^*,u)\leq\varepsilon$.
\begin{proof}[Proof of the claim.] Suppose that the contrary hold, that is $\gamma(p^*,u)>\varepsilon$. Then, from $(\bullet\bullet)$, there exists an integer $N$ such that for every $n\geq N$, we have that $\gamma(p^*,p_n)>\varepsilon$. Using $(\bullet)$ we see that $\liminf_{n\rightarrow+\infty}  \langle p^*- p_n, x_\varepsilon \rangle\geq \varpi_C(\varepsilon)$, which is a contradiction since $\theta>0$.
\end{proof}
\vskip5mm
Now, let us set $x:=y-\lambda_{\varepsilon, \theta}\hat{x}_\varepsilon$. Using the formula of  $\lambda_{\varepsilon, \theta}$ with (\ref{lambda}) we get (since $x_\varepsilon\in B_X$)
\begin{eqnarray*}
\|x\|&\leq& \|y\|+\lambda_{\varepsilon, \theta}\\
     &\leq& \frac{\theta\varpi_C(\varepsilon)}{2D} + \lambda_{\varepsilon, \theta}\\
     &<&  \frac{\theta\varpi_C(\varepsilon)}{2D} + (1+\theta)(\varepsilon+\theta).
\end{eqnarray*}
We can choose and fix $\theta>0$ sufficiently small so that we have $\|x\|<2\varepsilon$. This ends the proof of the theorem.
\end{proof}
Similarily to Corollary \ref{CThm0}, using  Theorem \ref{corprinc}, we obtain the following  localization.
\begin{corollary} \label{CThm} Let $X$ be a Banach space and $C$ be a norm bounded subset of $X^*$.  Let $(C,d)$ is a complete metric space such that the identity $I_C: (C,d)\rightarrow (C, \textnormal{weak}^*)$ is continuous. Suppose that $(C,d)$ has the $w^*\mathcal{USP}$. Let $f: (C,d) \to \R\cup \lbrace +\infty \rbrace$ be a proper bounded from below and lower semi-continuous. Then, there exists $a>0$ such that for every $\varepsilon \in ]0,a]$ and every $p^*\in C$ such that $f(p^*)< \inf_C f + \varepsilon \varpi_C(\varepsilon)$ (where , $\varpi_C(\varepsilon)$ denotes the modulus of the $w^*\mathcal{USP}$ of $C$), there exists $x\in X$ and  $u \in C$ such that

$(i)$ $d(p^*, u)\leq \varepsilon$,

$(ii)$ $\|x\| < 2 \varepsilon,$

$(iii)$ $f+\hat {x}$ attains  its strong minimum on $C$ at  $u$. 
\end{corollary}
\section{Porosity of the set of norm nonattaining operators}\label{S3}

Let $(K,d)$ be a complete metric space, $Y$ be a Banach space and $S_{Y^*}$ be the unit sphere of its dual.  By $C_b(K,Y)$, we denote the Banach space of all $Y$-valued bounded continuous functions equipped with the sup-norm. For every $(x, y^*)\in K\times S_{Y^*}$, we define the evaluation maps $\delta_x: T\mapsto T(x)$  and $y^*\circ \delta_x: T \mapsto \langle y^*, T(x) \rangle$, for all $T\in C_b(K,Y)$.  For any Banach space  $(Z, \|\cdot\|_Z)$ included in  $(C_b(K,Y)$ and such that $\|\cdot\|_Z\geq \|\cdot\|_{\infty}$, we have that $y^*\circ \delta_x \in Z^*$ for each $(x, y^*)\in L\times S_{Y^*}$. We suppose that the space $Z$ satisfies the following identity
\begin{eqnarray}\label{I}
 y^*_1\circ \delta_{x_1} =y^*_2\circ \delta_{x_2} \textnormal{ on } Z \Longrightarrow x_1=x_2 \textnormal{ and } y^*_1=y^*_2.
\end{eqnarray}
Let  $C_K:=\lbrace y^*\circ \delta_x: x\in K, y^*\in S_{Y^*}\rbrace\subset Z^*$. We define the complete pseudometric on $C_K$  as follows:
$$\gamma_\mathcal{P}(y^*\circ \delta_x, z^*\circ \delta_{x'}):=d(x,x'); \hspace{1mm} \forall y^*\circ \delta_x, z^*\circ \delta_{x'}\in C_K.$$

\begin{lemma} \label{property1} Let $(K,d)$ be a complete metric space, $Y$ be Banach spaces and $(Z,\|\cdot\|_Z)$ be a Banach space included in $C_b(K,Y)$ and satisfying:  

$(a)$ $\|.\|_Z\geq \|.\|_{\infty}$.

$(b)$ For every $\varepsilon>0$ there exists $\varpi_K(\varepsilon)>0$ and a collection $\lbrace b_{x,\varepsilon}: x\in K \rbrace\subset C_b(K,\R)$ such that, for every $e\in S_Y$ and every $x\in K$, we have that $b_{x,\varepsilon}.e\in Z$, $(\|b_{x,\varepsilon}\|_{\infty}\leq) \|b_{x,\varepsilon}.e\|_Z\leq 1$  and
\begin{eqnarray} \label{bump}
b_{x,\varepsilon}(x)-\varpi_K(\varepsilon)\geq \sup_{x'\in K:d(x',x)\geq \varepsilon} |b_{x,\varepsilon}(x')|.
\end{eqnarray}
Then, $Z$ satisfies the identity (\ref{I}) and the set $(C_K, \gamma_\mathcal{P})$  is a complete pseudometric space having the $w^*\mathcal{USP}$ in $Z^*$.
\end{lemma}
\begin{proof}  The fact that  $y^*\circ \delta_x \in Z^*$ for each $(x, y^*)\in K\times S_{Y^*}$, follows from part $(a)$. The map $\gamma_\mathcal{P}$ is well defined. Indeed, we prove that $y^*\circ \delta_x=z^*\circ \delta_{x'}$ implies that $x=x'$ and $y^*=z^*$. Let $e\in S_Y$ be such that $\langle y^*, e\rangle=\langle z^*, e\rangle$ ($e\in \textnormal{Ker}(y^*-z^*)$). Since $y^*\circ \delta_x(b_{x,\varepsilon}.e)=z^*\circ \delta_{x'}(b_{x,\varepsilon}.e)$ for every $\varepsilon >0$, it follows that $b_{x,\varepsilon}(x)=b_{x,\varepsilon}(x')$, for every $\varepsilon >0$, which implies that $x=x'$ by using the condition (\ref{bump}). Now,  we have that  $y^*\circ \delta_x(b_{x,\varepsilon}.e)=z^*\circ \delta_{x'}(b_{x,\varepsilon}.e)$ for every $e\in S_Y$. This implies (since $x=x'$) that $\langle y^*, e\rangle=\langle z^*, e\rangle$ for all $e\in S_Y$ and so $y^*=z^*$. Now, it is clear that $(C_K, \gamma_\mathcal{P})$  is a complete pseudometric space, since $(K,d)$ is a complete metric space. It remains to prove that $(C_K, \gamma_\mathcal{P})$  has the $w^*\mathcal{USP}$ is $Z^*$. Indeed, for every $y^*\in S_{Y^*}$ and $\varepsilon >0$, choose and fix an $e_{y^*,\varepsilon}\in S_Y$ such that $\langle y^*, e_{y^*,\varepsilon}\rangle >1-\frac{\varpi_{K}(\varepsilon)}{2(1+\varpi_{K}(\varepsilon))}>0$  and let us define for each  $(x, y^*)\in K\times S_{Y^*}$, the operator  $T_{(x, y^*,\varepsilon)}: X \to Y$ by $T_{(x, y^*,\varepsilon)}(x')=b_{x,\varepsilon}(x') e_{y^*,\varepsilon}$ for all $x'\in K$. By assumption,  $T_{(x, y^*,\varepsilon)}\in Z$ and $\|T_{(x_\lambda, y^*,\varepsilon)}\|_Z\leq 1$. On the other hand, for all $(x',z^*)\in K\times S_{Y^*}$ such that $d(x,x'):=\gamma_\mathcal{P}(y^*\circ \delta_x, z^*\circ \delta_{x'})\geq \varepsilon$, we have that: 
\begin{eqnarray*}
\langle y^*\circ \delta_x, T_{(x,y^*,\varepsilon)} \rangle - \frac{\varpi_K(\varepsilon)}{2}&=& \langle y^*, T_{(x,y^*,\varepsilon)}(x) \rangle - \frac{\varpi_{C_K}(\varepsilon)}{2}\\
&\geq& \langle y^*, T_{(x,y^*,\varepsilon)}(x) \rangle - \varpi_{C_K}(\varepsilon)\langle y^*, e_{y^*,\varepsilon}\rangle+\frac{\varpi_{C_K}(\varepsilon)}{2(1+\varpi_{C_K}(\varepsilon))}\\
&=& [b_{x,\varepsilon}(x)\langle y^*, e_{y^*,\varepsilon} \rangle- \varpi_{C_K}(\varepsilon)\langle y^*, e_{y^*,\varepsilon} \rangle]  +\frac{\varpi_{C_K}(\varepsilon)}{2(1+\varpi_{C_K}(\varepsilon))}\\
&\geq&   |b_{x,\varepsilon}(x')|\langle y^*, e_{y^*,\varepsilon}\rangle+\frac{\varpi_{C_K}(\varepsilon)}{2(1+\varpi_{C_K}(\varepsilon))}\\
&\geq& |b_{x,\varepsilon}(x')|(1 - \frac{\varpi_{C_K}(\varepsilon)}{2(1+\varpi_{C_K}(\varepsilon))}) +\frac{\varpi_{C_K}(\varepsilon)}{2(1+\varpi_{C_K}(\varepsilon))}\\
&\geq&  |b_{x,\varepsilon}(x')|, (\textnormal{ since },\|b_{x,\varepsilon}\|_{\infty}\leq 1)\\
&\geq& |b_{x,\varepsilon}(x')| |\langle z^*  ,  e_{y^*,\varepsilon} \rangle|\\
&=& |\langle z^*, T_{(x,y^*,\varepsilon)}(x') \rangle|\\
&=& |\langle z^*\circ \delta_{x'}, T_{(x,y^*,\varepsilon)} \rangle|\geq \langle z^*\circ \delta_{x'}, T_{(x,y^*,\varepsilon)} \rangle.
\end{eqnarray*}
It follows that $(C_K,\gamma_\mathcal{P})$ has the $w^*\mathcal{USP}$ is $Z^*$. 
\end{proof}
\begin{example}\label{NA-lin}
 Let $X$ be a Banach space with the property $(\alpha)$ and $Y$ be any Banach space.  Let $\lbrace x_\lambda: \lambda \in \Lambda \rbrace$, $\lbrace x^*_\lambda: \lambda \in \Lambda \rbrace$, subsets of $X$ and $X^*$ respectively, satisfying property $(\alpha)$. Let us set $K:=\overline{\lbrace x_\lambda: \lambda \in \Lambda \rbrace}^{\|\cdot\|}$. It is easy to see, thanks to parts $1)$ and $2)$ of property $(\alpha)$ (see Example B in Section \ref{S1}), that for every $\varepsilon>0$ there exists $\varpi_K(\varepsilon)>0$ such that,  for every $x\in K$, there exists $b_{x,\varepsilon}\in \lbrace x^*_\lambda: \lambda \in \Lambda \rbrace\subset S_{X^*}$ such that 
\begin{eqnarray*} 
\langle b_{x,\varepsilon},x\rangle -\varpi_K(\varepsilon)\geq \sup_{x'\in K:\|x'-x\|\geq \varepsilon} |\langle b_{x,\varepsilon}, x'\rangle|.
\end{eqnarray*}
Thus, every closed subspace $R(X,Y)$ of $B(X,Y)\subset (C_b(K,Y),\|\cdot\|_{\infty})$, contaninig $F(X,Y)$, satisfies $(a)$ and $(b)$ of Lemma \ref{property1}.
\end{example}
\begin{example} \label{Cont} By $C_b^u(K,Y)$, we denote the Banach space of all bounded uniformly continuous operators from a complete metric space $(K,d)$ to a Banach space $Y$ equipped with the sup-norm. It is easy to see that $(C_b^u(K,Y), \|\cdot\|_{\infty})$ and $(C_b(K,Y), \|\cdot\|_{\infty})$ satisfies $(a)$ and $(b)$ of Lemma \ref{property1}, with $b_{x,\varepsilon}: z \mapsto \max (0, 1-\frac{d(z,x)}{\varepsilon})$.
\end{example}
\begin{example}\label{Diff}
  Let $X$ be a Banach space such that there exists a  Lipschitz $C^1$-bump function from $X$ into $\R$ and  $Y$ be a Banach space. By $C_b^1(X,Y)$, we denote the Banach space of  all bounded continuously Fr\'echet differentiable functions from $X$ to $Y$ equipped with the norm: for all $f\in C_b^1(X,Y)$
$$\|f\|:=\max(\|f\|_{\infty}, \|f'\|_{\infty}).$$
The above Proposition applies to $Z=C_b^1(X,Y)$. 
\end{example}
\begin{remark} \label{Rbump} In the nonlinear operators case, the hypothesis in the formula $(\ref{bump})$ of Lemma \ref{property1} can be replaced by the following strong but fairly general and useful condition (the existence of "bump function" in $Z$):  For every $\varepsilon>0$ there exists a collection $\lbrace b_{x,\varepsilon}: x\in K, \rbrace\subset C_b(K,Y)$ such that, $b_{x,\varepsilon}.e\in Z$ and $\|b_{x,\varepsilon}.e\|\leq 1$, for every $e\in S_Y, x\in K, \varepsilon>0$ and satisfying:
\begin{eqnarray*} 
b_{x,\varepsilon}\geq 0; \hspace{1mm} b_{x,\varepsilon}(x)=1 \textnormal{ and  }b_{x,\varepsilon}(y)=0, \textnormal{ whenever } d(y,x)\geq \varepsilon.
\end{eqnarray*}
\end{remark}
 
A general and abstract statement on operator (linear or not) attaining their sup-norm is given in the following result. Lemma \ref{property1} gives a general criterion for which the following theorem applies. Example \ref{NA-lin},  Example \ref{Cont} and  Example \ref{Diff} are particular cases. 

\begin{theorem} \label{NA} Let  $(K,d)$ be a complete metric space and  $Y$ be a Banach space. Let $(Z, \|\cdot\|_Z)$ be a Banach space included in $C_b(K,Y)$ such that $\|\cdot\|_Z\geq \|\cdot\|_{\infty}$ and satisfying the identity $(\ref{I})$. Suppose that $(C_K,\gamma_\mathcal{P})$ is a complete pseudometric space having the $w^*\mathcal{USP}$ is $Z^*$. Then, for every $h \in C_b(K,Y)$,  the set
\begin{eqnarray*}
\mathcal{N}(h):=\lbrace g\in Z: h+g \textnormal{ does not attains strongly its sup-norm } \rbrace,
\end{eqnarray*}
is  a $\sigma$-porous subset of $(Z,\|\cdot\|_Z)$.

\noindent Moreover, the following “quantitative version” of the Bishop-Phelps-Bollobás theorem holds: for every $\varepsilon>0$, there exists $\lambda(\varepsilon) >0$ such that for every $f\in Z$, $\|f\|_{\infty}=1$ and every $x\in K$ satisfying $\|f(x)\|>1-\lambda(\varepsilon),$ there exists $k \in Z$, $\|k\|_{\infty}=1$ and $\overline{x}\in K$ such that 

$(i)$ $x\mapsto \|k(x)\|$ attains strongly  its maximum on $K$  at  $\overline{x}$, 

$(ii)$ $d(\overline{x},x)<\varepsilon$ and $\|f-k\|_{\infty}<\varepsilon.$
\end{theorem}

\begin{proof} Since $(C_K,\gamma_\mathcal{P})$  is a complete pseudometric space having the $w^*\mathcal{USP}$ in $Z^*$, by applying Theorem \ref{princ} (changing "infinimum" by "supremum") to $(C,\gamma_\mathcal{P})$ with the function 
\begin{eqnarray*}
\hat{h}: C_K &\to& \R\\
               y^*\circ \delta_x &\mapsto& \langle y^*, h(x)\rangle,
\end{eqnarray*}
we get  a $\sigma$-porous subset $\mathcal{N}(h)$  of $Z$ such that for every $f\in  Z\setminus \mathcal{N}(h)$, we have that $\hat{h}+\hat{f}$ attains $\gamma_\mathcal{P}$-strongly-directionally its supremum over $C_K$ at some direction $y_f^* \circ \delta_{x_f}\in C_K$. This implies that, for every $f\in Z\setminus \mathcal{N}(h)$, the function $\|(h+f)(\cdot)\|$ attains strongly its maximum on $K$  at  $x_f\in K$. Indeed, let $(u_n)\subset K$ such that $\|(h+f)(u_n)\|\to \|h+f\|_{\infty}$. 
 By the Hahn-Banach theorem, there exists  $(y^*_n)\subset S_{Y^*}$ such that $\|(h+f)(u_n)\| =\langle y^*_{n}, (h+f)(u_n)\rangle$. Thus  $\langle y^*_{n}\circ \delta_{u_n}, h+f\rangle\to \|h+f\|_{\infty}=\sup_{x\in K, y^*\in S_{Y}^*}\langle y^* \circ \delta_{x}, h+f\rangle$, which implies that  $d(u_n,x_f):=\gamma_\mathcal{P}(y^*_n \circ \delta_{u_n}, y^*_f \circ \delta_{x_f})\to 0$, since $\hat{h}+\hat{f}$ attains $\gamma_\mathcal{P}$-strongly-directionally its supremum over $C_K$ at  $y_f^* \circ \delta_{x_f}$. By the continuity of $\|(h+f)(\cdot)\|$, we have that $\|h+f\|_{\infty}=\|(h+f)(x_f)\|$. Hence, $\|(h+f)(\cdot)\|$ attains strongly its maximum on $K$ at $x_f$.

 The second part of the theorem, follows from Theorem ~\ref{corprinc}. Indeed,  let $\varepsilon>0$, $\lambda(\varepsilon)=\frac{\varepsilon}{4} \varpi_{C_K}(\varepsilon/4)>0$, (where $\varpi_{C_K}$ is the modulus of uniform $w^*\mathcal{USP}$ of $(C_K,\gamma_\mathcal{P})$ in $Z^*$).  Let, $f\in Z$, $\|f\|_{\infty}=1$ and $x\in K$ such that 
\begin{eqnarray*}
\|f(x)\|>1-\frac{\varepsilon}{4} \varpi_{C_K}(\varepsilon/4)&=&\|f\|_{\infty}-\frac{\varepsilon}{4} \varpi_{C_K}(\varepsilon/4). 
\end{eqnarray*}
We have that $1=\|f\|_{\infty}=\sup_{y^*\circ \delta_z\in C_K}\langle y^*\circ \delta_z, f \rangle$. Moreover, there exists by the Hanh-Banach theorem an $y^*_x\in S_{Y^*}$ such that $$\langle y^*_x\circ \delta_x, f\rangle := \langle y^*_x, f(x)\rangle= \|f(x)\|.$$ Thus, the above inequality can be writen as follows:
\begin{eqnarray*}
\langle y^*_x\circ \delta_x, f\rangle > \sup_{y^*\circ \delta_z\in C_K}\langle y^*\circ \delta_z, f \rangle-\frac{\varepsilon}{4} \varpi_{C_K}(\varepsilon/4). 
\end{eqnarray*}
 We apply Theorem ~\ref{corprinc}, with the function $\hat{f}$ (changing the "infinimum by the "supremum") with the set $C_K$ to obtain some $g\in Z$ and a point $\overline{y}^* \circ \delta_{\overline{x}}\in C_K$ such that

 $(a)$ $\gamma_\mathcal{P}(\overline{y}^* \circ \delta_{\overline{x}}, y^*_x\circ \delta_x):=d(\overline{x},x)<\frac{\varepsilon}{4},$ 

$(b)$ $\|g\|_{\infty}\leq \|g\|_{Z}<\frac{\varepsilon}{2}$, 

$(c)$ $\hat{f}-\hat{g}$ attains $\gamma_\mathcal{P}$-strongly-directionally its supremum over $C_K$ at the point $\overline{y}^* \circ \delta_{\overline{x}}$. This leads, as we have shown above, that $\|(f-g)(\cdot)\|$ attains strongly its maximum at $\overline{x}$.  Equivalently, the function $k:=\frac{f-g}{\|f-g\|_{\infty}}$ is such that $\|k(\cdot)\|$ attains strongly its maximum on $K$ at $\overline{x}$,  $\|k\|_{\infty}=1$  and we have (using triangular inequality), 
\begin{eqnarray*}
\|f-k\|_{\infty}=\|f- \frac{f-g}{\|f-g\|_{\infty}}\|_{\infty} &=& \|g+ (f-g-\frac{f-g}{\|f-g\|_{\infty}})\|_{\infty}\\
                          &\leq& 2 \|g\|_{\infty}\\
                          &<& \varepsilon. 
\end{eqnarray*}
This concludes the proof.
\end{proof}

As a direct consequence of Theorem \ref{NA} and Lemma \ref{property1}, we obtain the following result on norm attaining linear operators, which generalizes some old results, passing from the density to the complement of a $\sigma$-porous set and from norm attained to strongly norm attained.

\begin{corollary} \label{Ethm} Let $X$ be a Banach space having property $(\alpha)$.  Then,  for every  Banach space $Y$, every  $S\in B(X,Y)$ and  every closed subspace $R(X,Y)$ of $B(X,Y)$ contaning $F(X,Y)$, we have that the set 
\begin{eqnarray*}
\mathcal{N}(S):=\lbrace T\in R(X,Y):  S+T \textnormal{ does not attains strongly its norm}\rbrace,
\end{eqnarray*}
is a $\sigma$-porous subset of $R(X,Y)$. In particular (with $S=0$), we have that $NAR(X,Y)$  is the complement of a $\sigma$-porous subset of $R(X,Y)$.
\end{corollary}

\begin{proof}  Let $\lbrace x_\lambda: \lambda \in \Lambda \rbrace$, $\lbrace x^*_\lambda: \lambda \in \Lambda \rbrace$, subsets of $X$ and $X^*$, satisfying property $(\alpha)$. Let us set $K:=\overline{\lbrace x_\lambda: \lambda \in \Lambda \rbrace}^{\|\cdot\|_X}$. It is easy to see, thanks to property $(\alpha)$ (see Example B in section \ref{S1}), that for every $\varepsilon>0$ there exists $\varpi_K(\varepsilon)>0$ such that,  for every $x\in K$, there exists $b_{x,\varepsilon}\in \lbrace x^*_\lambda: \lambda \in \Lambda \rbrace\subset S_{X^*}$ such that 
\begin{eqnarray*} 
\langle b_{x,\varepsilon},x\rangle -\varpi_K(\varepsilon)\geq \sup_{x'\in K:\|x'-x\|\geq \varepsilon} |\langle b_{x,\varepsilon}, x'\rangle|.
\end{eqnarray*}
On the other hand, since the absolute convex hull of the set $\lbrace x_\lambda: \lambda \in \Lambda \rbrace$ is dense in the unit ball of $X$, we have that for every $T\in B(X,Y)$, $$\|T\|=\sup_{x\in K}\|T(x)\|.$$ 
Considering $Z:=R(X,Y)$ as a closed subspace of $(C_b(K,Y),\|\cdot\|_{\infty})$, it is clear that $R(X,Y)$ satisfies parts $(a)$ and $(b)$ of Lemma \ref{property1}. Thus, the conclusion follows from  Theorem \ref{NA}.
\end{proof}
\begin{corollary} \label{Ecor}  Let $X$ be a Banach space having property $(\alpha)$.  Let $(T_n)\subset B(X,Y)$ be a sequence of bounded linear operators. Then, for every $\varepsilon>0$, there exists a compact operator $T$ which is norm-limit of a sequence of finite-rank operators, such that $\|T\|<\varepsilon$ and $T_n+T$ attains strongly its norm for every $n\in \N$.
\end{corollary}
\begin{proof}  We apply Theorem \ref{Ethm} with $R(X,Y)=\overline{F(X,Y)}$ and $S=T_n$ for each $n\in \N$, we get  $\sigma$-porous sets $\mathcal{N}(T_n)$ such that every $T\in \overline{F(X,Y)}\setminus \mathcal{N}(T_n)$, $T_n +T$ attains strongly its norm. The set $\cup_n \mathcal{N}(T_n)$ is also a  $\sigma$-porous set. Thus, in particular, $\overline{F(X,Y)}\setminus \cup_n \mathcal{N}(T_n)$ is dense in $\overline{F(X,Y)}$. Hence, for every $\varepsilon >0$, there exists $T\in  \overline{F(X,Y)}\setminus \cup_n \mathcal{N}(T_n)$ such that $\|T\|<\varepsilon$ and $T_n +T$ attains strongly its norm for all $n\in \N$.
\end{proof}
Since $(C_b^u(K,Y), \|\cdot\|_{\infty})$  satisfies $(a)$ and $(b)$ of Lemma \ref{property1}, with $b_{x,\varepsilon}: z \mapsto \max (0, 1-\frac{d(z,x)}{\varepsilon})$, using Theorem \ref{NA} we immediately obtain the following result.
\begin{corollary} \label{Ecor1} Let $(K,d)$ be a complete metric space and $Y$ be a Banach space. Then, the subset of $C_b(K,Y)$ (resp. of $C_b^u(K,Y)$) of all bounded continuous (resp. uniformly continuous) operators attaining strongly their sup-norm, is a complement of a $\sigma$-porous subset of the Banach space $(C_b(K,Y), \|\cdot\|_{\infty})$ (resp. of $(C_b^u(K,Y), \|\cdot\|_{\infty})$).  
\end{corollary}


\bibliographystyle{amsplain}

\end{document}